\documentclass[11pt]{article}

\usepackage{enumerate}
\usepackage{amsmath}
\usepackage{amssymb,latexsym}
\usepackage{amsthm}
\usepackage{color}
\usepackage{graphicx}
\usepackage[notcite, notref]{}
\usepackage{amscd}
\usepackage{hyperref}

\title{Linear sets in the projective line over the endomorphism ring of a finite field\thanks{This
work was supported by GNSAGA of \emph{Istituto Nazionale di Alta Matematica ``F.~Severi''} (Rome) and partly done while the first
author was Visiting Professor at the University of Padua, Vicenza, Italy.}}
\author{Hans Havlicek \and Corrado Zanella}

\newcommand{\cR}{{\mathcal R}}

\newcommand{\cB}{{\mathcal B}}

\newcommand{\cG}{{\mathcal G}}

\newcommand{\cF}{{\mathcal F}}
\newcommand{\cP}{{\mathcal P}}
\newcommand{\cH}{{\mathcal H}}
\newcommand{\cD}{{\mathcal D}}
\newcommand{\cU}{{\mathcal U}}

\newcommand{\cQ}{{\mathcal Q}}
\newcommand{\cS}{{\mathcal S}}

\newcommand{\F}{{\mathbb F}}
\newcommand{\Fq}{\F_q}
\newcommand{\Fqt}{\F_{q^t}}

\newcommand{\la}{\langle}
\newcommand{\ra}{\rangle}

\newtheorem{theorem}{Theorem}[section]
\newtheorem*{theorem*}{Theorem}
\newtheorem{lemma}[theorem]{Lemma}

\newtheorem{proposition}[theorem]{Proposition}

\theoremstyle{definition}
\newtheorem{definition}[theorem]{Definition}
\newtheorem{example}[theorem]{Example}

\newtheorem{remark}[theorem]{Remark}

\DeclareMathOperator{\PG}{{PG}}

\DeclareMathOperator{\GL}{{GL}}
\DeclareMathOperator{\PGL}{{PGL}}
\DeclareMathOperator{\PGaL}{P\Gamma L}
\DeclareMathOperator{\Gal}{Gal}

\DeclareMathOperator{\SL}{{SL}}
\DeclareMathOperator{\PSL}{{PSL}}

\usepackage{bbm}
\DeclareMathOperator{\End}{End}
\DeclareMathOperator{\diag}{diag}
\newcommand{\nodist}{\mathrel{\not\kern-.33em{\bigtriangleup}}}
\newcommand{\one}{\mathbbm 1}

\newcommand{\gal}{\Gal(\Fqt/\Fq)}
\DeclareMathOperator{\Spec}{Spec}

\newcommand{\dist}{\mathrel{\bigtriangleup}}
\DeclareMathOperator{\Char}{Char}
\newcommand{\Fqtt}{\F_{q^{2t}}}

\begin{document}

\maketitle

\begin{abstract}
  Let $\PG(1,E)$ be the projective line over the endomorphism ring $
  E=\End_q(\Fqt)$ of the $\Fq$-vector space $\Fqt$. As is well known there is a
  bijection $\Psi:\PG(1,E)\rightarrow\cG_{2t,t,q}$ with the Grassmannian of the
  $(t-1)$-subspaces in $\PG(2t-1,q)$. In this paper along with any $\Fq$-linear
  set $L$ of rank $t$ in $\PG(1,q^t)$, determined by a $(t-1)$-dimensional
  subspace $T^\Psi$ of $\PG(2t-1,q)$, a subset $L_T$ of $\PG(1,E)$ is investigated. Some properties of
  linear sets are expressed in terms of the projective line over the
  ring $E$. In particular the attention is focused on the relationship between
  $L_T$ and the set $L'_T$, corresponding via $\Psi$ to a collection of
  pairwise skew $(t-1)$-dimensional subspaces, with $T\in L'_T$, each of which
  determine $L$. This leads among other things to a characterization of the
  linear sets of pseudoregulus type. It is proved that a scattered linear set $L$ related to
  $T\in\PG(1,E)$ is of pseudoregulus type
  if and only if there exists a projectivity $\varphi$ of $\PG(1,E)$ such that $L_T^\varphi=L'_T$.
\end{abstract}

 \section{Introduction}\label{s:intro}

\subsection{Motivation}

In this paper linear sets of rank $t$ in the projective line $\PG(1,q^t)$ are
investigated, where $q$ is a power of a prime $p$. Such linear sets can be
described by means of the \emph{field reduction map} $\cF=\cF_{2,t,q}$
\cite{LaVa2015} mapping any point $\la (a,b)\ra_{q^t}\in
\PG_{q^t}(\Fqt^2)\cong\PG(1,q^t)$ to the $(t-1)$-subspace\footnote{Abbreviation
for $(t-1)$-dimensional subspace.} of $\PG_q(\Fqt^2)\cong\PG(2t-1,q)$
associated with $\la (a,b)\ra_{q^t}$ (considered here as a $t$-dimensional
$\Fq$-vector subspace). A point set $L \subseteq \PG(1,q^t)$ is said to be
\emph{$\F_q$-linear} (or just \emph{linear}) \emph{of rank $n$} if $L=\cB(T')$,
where $T'$ is an $(n-1)$-subspace of $\PG_q(\Fqt^2)$, and
\begin{equation}\label{e:B(T)}
      \cB(T')=\left\{\la (u,v) \ra_{q^t}\mid \la (u,v)\ra_q\in T'\right\}
        =\left\{P\in\PG(1,q^t)\mid P^\cF\cap T'\neq\emptyset\right\}.
\end{equation}
Additionally, each such $T'$ gives rise to the set $\cU(T')=\cB(T')^\cF=L^\cF$,
which is a collection of $(t-1)$-subspaces belonging to the \emph{standard
Desarguesian spread} $\cD=\PG(1,q^t)^\cF$ of $\PG_q(\Fqt^2)$.

If a linear set $L$ of rank $n$ in $\PG(1,q^t)$ has size
$\theta_{n-1}=(q^n-1)/(q-1)$ (which is the maximum size for a linear set of
rank $n$), then $L$ is a \emph{scattered} linear set. For generalities on the
linear sets the reader is referred to \cite{LaVdV2015}, \cite{LaVa2015},
\cite{LaZa2015}, \cite{LaZa201X}, and \cite{Po2010}.

As it has been pointed out in \cite[Prop.~2]{LaShZa2015}, if $L=\cB(T')$ is a
scattered linear set of rank $t$ in $\PG(1,q^t)$, then the union of all
subspaces in $\cU(T')=L^\cF$ is a hypersurface $\cQ$ of degree $t$ in
$\PG(2t-1,q)$, and an embedded product space isomorphic to
$\PG(t-1,q)\times\PG(t-1,q)$. So, $\cQ$ has two partitions in
$(t-1)$-subspaces. The first one is $\cU(T')$, the second one is
$\cU'(T')=\{T'h\mid h\in\Fqt^*\}$, where $T'h=\{\la (hu,hv)\ra_q\mid\la
(u,v)\ra_q\in T'\}$. By Prop.~\ref{p:caratt-massimali}, the family $\cU'(T')$
can be recovered uniquely from $\cU(T')$ and $T'$ (disregarding that $\Fqt^2$
is the underlying vector space of our $\PG(2t-1,q)$).

For $t=n$ there is alternative approach to $\cB(T')$ and $\cU(T')$ irrespective
of whether $T'$ is scattered or not. It is based on the $\Fq$-{endomorphism}
ring $E$ of $\Fqt$ and the projective line $\PG(1,E)$ over this ring. On the
one hand, there is a bijection $\Psi$ between the projective line $\PG(1,E)$
and the Grassmannian $\cG_{2t,t,q}$ of $(t-1)$-subspaces of $\PG_q(\Fqt^2)$.
So, instead of $T'$ we may consider its image under $\Psi^{-1}$, which is
\emph{a point $T$ of\/ $\PG(1,E)$}. On the other hand, we have a natural
embedding $\iota:\PG(1,q^t)\rightarrow\PG(1,E)$. It maps the linear set
$\cB(T')$ to \emph{a subset $\cB(T')^\iota=:L_T$ of\/ $\PG(1,E)$}, which in
turn is the preimage under $\Psi$ of $\cU(T')$. In Section \ref{s:basic}, we
take up these ideas, but we start with an equivalent definition, which is in
terms of $\PG(1,E)$ only, of the set $L_T$. There we also define a second set
$L'_T\subset\PG(1,E)$ in such a way that $(L'_T)^\Psi$ equals the set
$\cU'(T^\Psi)$ from above in the scattered case. Furthermore, since $T$ will
play a predominant role, $\cB(T')=\cB(T^\Psi)$ will frequently also be denoted
by $\cB(T)$; \emph{mutatis mutandis} this applies also to $\cU(T')$ and
$\cU'(T')$.

A special example of a scattered linear set $L=\cB(T)$ in $\PG(1,q^t)$ is a
\emph{linear set of pseudoregulus type}, defined in \cite{DoDu2014},
\cite{LuMaPoTr2014}, and further investigated in \cite{CsZa2016}. In our
setting it is obtained by taking $T=E(\one,\tau)$, where $\tau$ is a generator
of the Galois group $\gal$. The related hypersurface $\cQ$ in $\PG(2t-1,q)$ has
been studied in \cite{LaShZa2015}, revealing a high degree of symmetry. As a
matter of fact there are $t$ families $\cS_0$, $\cS_1$, $\ldots$, $\cS_{t-1}$
of $(t-1)$-subspaces partitioning $\cQ$ \cite[Thm.~6]{LaShZa2015} where
$\cS_0=\cU(T)=L_T^\Psi$ and $\cS_1=\cU'(T)=(L'_T)^\Psi$ are defined above.
Furthermore, in \cite[Cor.~18]{LaShZa2015} it is proved that the stabilizer of
$\cQ$ inside $\PGaL_{2t}(q)$ contains a dihedral subgroup of order $2t$ acting
on such $t$ families of $(t-1)$-subspaces. A consequence thereof is the
following result, for which we give a short direct proof in
Prop.~\ref{p:zero-converse}: \emph{There is a projectivity of\/ $\PG(1,E)$
mapping $L_T$ onto $L'_T$.} For $t\ge3$ this turns out to be a characteristic
property of the linear sets of pseudoregulus type (Thm.~\ref{t:pseudoreg}). As
the projectivities of $\PG(1,E)$ and the projectivities of $\PG_q(\Fqt^2)\cong
\PG(2t-1,q)$ are in one-to-one correspondence, this leads to the following:
\begin{theorem*}
  Let $L=\cB(T')$ be a scattered linear set of rank $t$ in $\PG(1,q^t)$, with
  $T'$ a $(t-1)$-dimensional subspace of $\PG(2t-1,q)$, and $t\ge3$. Then $L$
  is a linear set of pseudoregulus type if, and only if, a projectivity of\/
  $\PG(2t-1,q)$ exists mapping the first family $\cU(T')$ of subspaces of the
  related embedded product space to the second one $\cU'( T')$.
\end{theorem*}

\subsection{Notation}\label{ss:notation}

Let $E=\End_q(\Fqt)$ with $t\ge 2$ be the ring of $\Fq$-linear endomorphisms of
$\Fqt$. The ring $E$ has the identity $\one\in E$ as its unit element. The
multiplicative group comprising all invertible elements of $E$ will be denoted
as $E^*$.

Let us briefly recall the definition of the \emph{projective line over the ring
$E$}, which will be denoted by $\PG(1,E)$, and several basic notions; see
\cite[1.3]{BlHe2005}, \cite[3.2]{Ha2012}, and \cite[1.3]{He1995}. We start with
$E^2$, which is regarded as a \emph{left} module over $E$ in the usual way.
Elements of $E^2$ are written as rows. This module has the standard basis
$\left( (\one,0), (0,\one)\right)$, and so it is a free module of rank $2$. All
invertible $2\times 2$ matrices with entries in $E$ constitute the general
linear group $\GL_2(E)$, which acts in a natural way on the elements of $E^2$
from the \emph{right hand side}. Now $\PG(1,E)$, whose elements will be called
\emph{points}, is defined as the orbit of the cyclic submodule $E(\one,0)$ (the
``starter point'') under the action of the group $\GL_2(E)$ on $E^2$.
Therefore, any point of $\PG(1,E)$ can be written in the form
$E(\alpha,\beta)$, where the pair $(\alpha,\beta)\in E^2$ is \emph{admissible},
i.e., it is the first row of a matrix from $\GL_2(E)$. Furthermore, if
$(\alpha',\beta')$ is any element of $E^2$ then
$E(\alpha',\beta')=E(\alpha,\beta)$ holds precisely when there is an element
$\gamma\in E^*$ such that $(\alpha',\beta')=(\gamma\alpha,\gamma\beta)$. In
this case $(\alpha',\beta')$ is admissible too.

The projective line $\PG(1,E)$ is endowed with a binary \emph{distant relation}
$\dist$ as follows: The relation $\dist$ is the orbit of the pair $\left(
(\one,0), (0,\one)\right)$ under the (componentwise) action of $\GL_2(E)$. Thus
$E(\alpha,\beta)\dist E(\gamma,\delta)$ holds if, and only if, ${{\alpha\;
\beta}\choose {\gamma\; \delta}}\in\GL_2(E)$.

The map
\begin{equation}\label{e:Psi}
    \Psi: \PG(1,E) \to \cG_{2t,t,q} :
     E(\alpha,\beta)\mapsto
    \left\{\la(u^\alpha,u^\beta)\ra_q\mid u\in\Fqt^*\right\}
\end{equation}
is a bijection of $\PG(1,E)$ onto the Grassmannian $\cG_{2t,t,q}$ of
$(t-1)$-subspaces of $\PG_q(\Fqt^2)\cong\PG(2t-1,q)$. Any two points of
$\PG(1,E)$ are distant if, and only if, their images under $\Psi$ are disjoint
(or, said differently, complementary) \cite[Thm.~2.4]{Bl1999}. For versions of
the previous results in terms of matrix rings we refer to
\cite[10.2]{BlHe2005}, \cite[5.2.3]{Ha2012}, and \cite[4.5]{He1995}. See also
\cite[123ff.]{Wan1996}, even though the terminology used there is quite
different from ours.

Let $\varphi$ denote a projectivity of $\PG(1,E)$, i.e., $\varphi$ is given by
a matrix
\begin{equation}\label{e:tilde}
    \begin{pmatrix}
    \alpha & \beta\\ \gamma & \delta
    \end{pmatrix}\in\GL_2(E)
\end{equation}
acting on $E^2$. Then the mapping
\begin{equation}\label{e:hat-phi}
\hat\varphi : \PG_q(\Fqt^2) \to \PG_q(\Fqt^2) :
    \la(u,v)\ra_q\mapsto \la(u^\alpha+v^\gamma, u^\beta + v^\delta)\ra_q
\end{equation}
is a projective collineation. The action of $\hat\varphi$ on the Grassmannian
$\cG_{2t,t,q}$ is given by $\Psi^{-1}\varphi\Psi$. By \cite[642--643]{Lang93},
every projective collineation of $\PG_q(\Fqt^2)$ can be written as in
\eqref{e:hat-phi} for some matrix from $\GL_2(E)$.

Under any projectivity of $\PG(1,E)$ the distant relation $\dist$ is preserved.
The obvious counterpart of this observation is the fact that under any
projective collineation of $\PG_q(\Fqt^2)$ the complementarity of subspaces
from $\cG_{2t,t,q}$ is preserved.

If $a\in\Fqt$ then $\rho_a\in E$ is defined by $x^{\rho_a}=ax$ for all
$x\in\Fqt$. The mapping
\begin{equation*}
    \Fqt\to  E:a\mapsto \rho_a
\end{equation*}
is a monomorphism of rings taking $1\in\Fqt$ to the identity $\one\in E$. The
image of this monomorphism will be denoted by $F$. We now consider $\PG(1,F)$
as a subset of $\PG(1,E)$ by identifying $F(\rho_a,\rho_b)$ with
$E(\rho_a,\rho_b)$ for all $(a,b)\in \Fqt^2\setminus\{(0,0)\}$. This allows us
to embed the projective line $\PG(1,q^t)$ in the projective line $\PG(1,E)$ as
follows:
\begin{equation}\label{e:embed}
    \iota:\PG(1,q^t)\to\PG(1,  E ): \la(a,b)\ra_{q^t}\mapsto  E (\rho_a,\rho_b) .
\end{equation}
Following \cite{BlHa2000a}, the image of $\PG(1,F)$ under any projectivity of
$\PG(1,E)$ is called an \emph{$F$-chain}\footnote{Our $F$-chains are different
from the chains in \cite{BlHe2005} and \cite{He1995}, since $F$ is not
contained in the centre of $E$.} of $\PG(1,E)$. In particular,
$\PG(1,q^t)^\iota$ is an $F $-chain of $\PG(1,E)$.

Any two distinct points of $\PG(1,E)$ are distant precisely when they belong to
a common $F$-chain \cite[Lemma 2.1]{BlHa2000a}. From this we obtain the
following result \cite[Thm.~3.4.7]{Ha2012}, which is a slightly modified
version of \cite[Thm.~2.3]{BlHa2000a}:

\begin{proposition}\label{p:0}Given three distinct
points $P_1,Q_1,R_1$ on an $F$-chain $C_1$ and three distinct points $P_2,Q_2,R_2$ on an
$F$-chain $C_2$ there is at least one projectivity $\pi$ of $\PG(1,E)$ with
$P_1^\pi= P_2$, $Q_1^\pi=Q_2$, $R_1^\pi=R_2$ and $C_1^\pi=C_2$.
\end{proposition}

\section{Scattered points}\label{s:basic}

\begin{definition}\label{d:L(T)}
For any point $T=E(\alpha,\beta)\in\PG(1,E)$ define:
  \begin{eqnarray*}
    L_T&=&\left\{ E(\rho_a,\rho_b)\mid (a,b)\in(\Fqt^2)^*\mbox{ {s.t. }} E(\rho_a,\rho_b)\nodist T\right\};\\
    L'_T&=&\left\{T\cdot\diag(\rho_h,\rho_h) \mid h\in\Fqt^*\right\}.
  \end{eqnarray*}
\end{definition}
Also, we introduce the shorthand $Th:= T\cdot\diag(\rho_h,\rho_h)$, where $h$
is as above. By the proof of Prop.~\ref{p:hans3} below, the point set $L'_T$ is
the orbit of $T$ under the group of all projectivities of $\PG(1,E)$ that fix
$\PG(1,F)$ pointwise.

The following diagram describes the relationships involving some objects
defined so far. (Note that the right hand side of \eqref{e:B(T)} gives
$\cB(T)^\cF=L_T^\Psi$.)

\begin{center}
\includegraphics[scale=0.25]{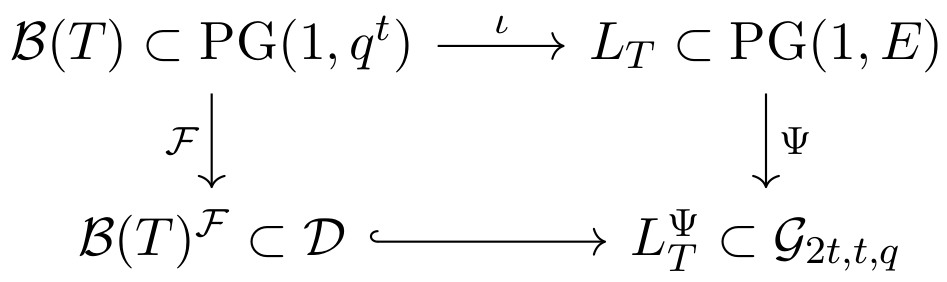}
\end{center}

\begin{definition}
  A \emph{scattered point} of $\PG(1,E)$ is a point $T$ such that
  $\#L_T=\theta_{t-1}$.
\end{definition}

A point $T\in\PG(1,E)$ is scattered if, and only if, $\cB(T)$ is a scattered
linear set. For a point $X=E(\rho_a,\rho_b)$ is distant from $T$ if, and only
if, the $(t-1)$-subspace $X^\Psi$ defined by the vector subspace
$\la(a,b)\ra_{q^t}$ is disjoint from $T^\Psi$.

\begin{example}\label{T_0}
If $\tau$ is a generator of $\gal$ and $T_0=E(\one,\tau)$, then $\cB(T_0)$ is a
scattered linear set of pseudoregulus type \cite{LuMaPoTr2014}. Hence $T_0$ is
a scattered point of $\PG(1,E)$.
\end{example}

\begin{definition}
  For $T\in\PG(1,E)$, the set $L_T$ will be said \emph{of pseudoregulus type}
  when $L_T^{\iota^{-1}}$ is a linear set of pseudoregulus type.
\end{definition}
Any two linear sets of pseudoregulus type are projectively equivalent
\cite{DoDu2014}, \cite{LuMaPoTr2014}. So $L_T$ is of pseudoregulus type if and
only if $L_T=L_{ E(\one,\tau)}^\pi$ where $\tau$ is a generator of $\gal$ and
$\pi$ is a projectivity of $\PG(1,F)$.
\begin{example}\label{T_1}
 The point $T_1= E(\one,\sigma\rho_g+\sigma^{t-1})$ with $\sigma:u\mapsto u^q$,
   $g\in\Fqt^*$, and $g^{\theta_{t-1}}\neq1$ is scattered
   \cite[Thm.~2]{LuPo2001}.
\end{example}

\begin{proposition}\label{p:Th}
  Let $T\in\PG(1,E)\setminus\PG(1,F)$. Then $Th=T k $ for any $h, k \in\Fqt^*$
  such that $h^{-1} k \in\Fq$. Furthermore, if $T$ is scattered, then
  $Th\bigtriangleup T k $ for any $h, k \in\Fqt^*$ such that
  $h^{-1} k \not\in\Fq$.
\end{proposition}
\begin{proof}
   If $h^{-1} k \in\Fq^*$ and $T=E(\alpha,\beta)$, then
  \[ Th=E(\alpha\rho_{h},\beta\rho_{h})=E(\rho_{h^{-1} k }\alpha\rho_{h},\rho_{h^{-1} k }\beta\rho_{h})=
  E(\alpha\rho_{ k },\beta\rho_{ k })=T k . \]

  Let $\cP'$ be the set of all points of $\PG_q(\Fqt^2)\cong\PG(2t-1,q)$
  belonging to the $(t-1)$-subspaces of ${L_T'}^\Psi$. By the previous
  paragraph, it follows $\#\cP'\le\theta_{t-1}^2$, and the equality holds if,
  and only if, for any $h, k $ the relation $h^{-1} k \in\Fqt\setminus\Fq$
  implies $Th\bigtriangleup T k $.

  Let $\cP$ be the set of all points of $\PG_q(\Fqt^2)$ belonging to the
  $(t-1)$-subspaces in $L_T^\Psi$. Then $\cP\subseteq \cP'$.

  If $T$ is scattered, then $\# \cP=\theta_{t-1}^2$.
\end{proof}
\begin{proposition}\label{p:varie}
  Let $T=E(\one,\beta)$ be a scattered point of $\PG(1,E)$. Then the following
  assertions hold:
  \begin{enumerate}[{\emph(}i{\emph)}]
    \item A point $P\in\PG(1,E)$ belongs to $L_T$ if, and only if, a
        $u\in\Fqt^*$ exists such that $P=E(\one,\rho_{u^\beta/u})$;
    \item the size of the set $I=\{u^\beta/u\mid u\in\Fqt^*\}$ is
        $\theta_{t-1}$;
    \item for any $u,v\in\Fqt^*$, $u$ and $v$ are $\Fq$-linearly dependent
        if, and only if, $u^\beta/u=v^\beta/v$;
    \item the dimension of $\ker\beta$ is at most one;
    \item $\beta$ is a singular endomorphism if, and only if, $
        E(\one,0)\in L_T$.
  \end{enumerate}
\end{proposition}
\begin{proof}
  Let $P=E(\rho_a,\rho_b)$ be a point. Then $P\in L_T$ holds precisely when the
  $(t-1)$-subspaces $P^\Psi$ and $T^\Psi$ are not disjoint; that is, there are
  two nonzero elements of $\Fqt$, say $u$ and $v$, such that $u=va$ and
  $u^\beta=vb$. This is equivalent to $a\neq0$ and $u^\beta=a^{-1}bu$. This
  implies $(i)$, and consequently $(v)$.

  The size of $I$ equals the size of $L_T$, and this implies $(ii)$.

  If $r\in\Fq^*$ and $u\in\Fqt^*$, then $(ru)^\beta/(ru)=u^\beta/u$. This
  implies that the size of the image of the map $u\in\Fqt^*\mapsto
  u^\beta/u\in\Fqt$ is at most $\theta_{t-1}$, and the equality holds only if
  condition $(iii)$ is satisfied. The last condition implies $(iv)$.
\end{proof}
Take notice that $(i)$ and $(v)$ hold irrespective of whether the point
$T\in\PG(1,E)$ is scattered or not.

The following result is merely a reformulation of \cite[Prop.~2]{LaShZa2015},
with $\beta\in E$ playing the role of the matrix $A$ from there.

\begin{proposition}\label{p:PGxPG-embed}
  Let $T=E(\one,\beta)$ be a scattered point. The map
  \begin{equation}\label{e:segremb}
        \varepsilon :\left(\la h\ra_q,\la(u,u^\beta)\ra_q\right)\mapsto\la(hu,hu^\beta)\ra_q
  \end{equation}
  is a projective embedding of the product space $\PG_q(\Fqt)\times T^\Psi$
  into $\PG_q(\Fqt^2)$, that is, is an injective mapping such that the image of
  any line of the product space is a line of $\PG_q(\Fqt^2)$.
\end{proposition}
\begin{remark}
In the case of non-scattered linear sets, the map $\varepsilon$ is not an
embedding, but the image of $\varepsilon$ is still a non-injective projection
of a Segre variety \cite{LaZa201X}.
\end{remark}

In \cite[Thm.~1]{Dy1991} a result similar to the following one is proved in
terms of the matrix group $\GL_2(q^t)$.
\begin{proposition}\label{p:hans1}%
Let $\kappa\in\PGaL_{2}(q^t)$ be a collineation of $\PG(1,q^t)$ whose
accompanying automorphism $\eta$ is in $\gal$. After embedding $\PG(1,q^t)$ in
$\PG(1,E)$ according to \eqref{e:embed}, the collineation
$\iota^{-1}\kappa\iota$ of $\PG(1,F )$ can be extended to at least one
projectivity of $\PG(1,E)$. Conversely, the restriction to $\PG(1,F)$ of any
projectivity of $\PG(1,E)$ that fixes $\PG(1,F)$ as a set is a collineation
with accompanying automorphism in $\gal$.
\end{proposition}
\begin{proof}
There is a matrix $(m_{ij})\in\GL_2(q^t)$ such that
\begin{equation*}
    \la (a,b)\ra_{q^t}\stackrel{\kappa}\mapsto \la (a^\eta,b^\eta)\ra_{q^t}\cdot (m_{ij})
    \mbox{~for all~} \la (a,b)\ra_{q^t}\in\PG(1,q^t).
\end{equation*}
For all $x,a\in \Fqt$ we have $xa^\eta = (x^{\eta^{-1}}a)^\eta$ and
so $\rho_{a^\eta}=\eta^{-1}\rho_a\eta$. The permutation of
$\PG(1,E)$ given by
\begin{equation}\label{e:extension}
     E(\alpha,\beta)\mapsto  E(\eta^{-1}\alpha\eta,\eta^{-1}\beta\eta)\cdot (\rho_{m_{ij}})
    \mbox{~for all~} E(\alpha,\beta) \in\PG(1,E)
\end{equation}
is a projectivity, since the automorphism of $E$ acting on $\alpha$ and $\beta$
is inner. By construction, this projectivity extends the collineation
$\iota^{-1}\kappa\iota$ of $\PG(1,F )$.
\par
Conversely, let $\pi$ be a projectivity of $\PG(1,E)$ that fixes $\PG(1,F )$ as
a set. Since $\PGL(2,q^t)$ acts (sharply) $3$-transitively on $\PG(1,q^t)$
there is a (unique) projectivity $\lambda$ of $\PG(1,q^t)$ such that the images
of $E(\one,0)$, $E(0,\one)$, $E(\one,\one)$ under $\iota^{-1}\lambda\iota$ and
$\pi$ are the same. We choose matrices $(\pi_{ij})\in\GL_2(E)$ and
$(c_{ij})\in\GL_2(q^t)$ that describe $\pi$ and $\lambda$, respectively. Then
$(\pi_{ij})\cdot(\rho_{c_{ij}})^{-1}$ induces a projectivity of $\PG(1,E)$ that
fixes each of the points $E(\one,0)$, $E(0,\one)$, $E(\one,\one)$ and also the
$F $-chain $\PG(1,F )$. So there is a $\delta\in E^*$ with
\begin{equation*}
    (\pi_{ij})=\diag(\delta,\delta)\cdot(\rho_{c_{ij}})
\end{equation*}
and for each $b\in\Fqt$ there is a unique $b'\in\Fqt$ such that
\begin{equation*}
     E(\one,\rho_b)\cdot\diag(\delta,\delta) =  E(\delta,\rho_b\delta)
    =  E(\one,\rho_{b'}) = E(\delta,\delta\rho_{b'}).
\end{equation*}
This leads us to $\delta^{-1}\rho_b\delta= \rho_{b'}$ for all $b\in\Fqt$. Thus
the inner automorphism of $E$ given by $\delta$ restricts to an automorphism of
the field $F $. Going back to $\Fqt$ shows that $\eta:\Fqt\to\Fqt:b\mapsto b'$
is an automorphism of $\Fqt$. Furthermore, we read off $\eta\in\gal$ from
$\rho_b$ being in the centre of $E$ for all $b\in\Fq$.
\par
Let $d:=1^\delta\in\Fqt^*$ and choose any $b\in\Fqt$. Calculating the image of
$d$ under $\delta^{-1}\rho_b\delta= \rho_{b^\eta}$ in two ways gives $b^\delta
= d b^\eta $, whence $\delta=\eta\rho_d$. This leads us finally to
\begin{equation}\label{e:mu_final}
    (\pi_{ij})=\diag(\eta,\eta)\cdot(\rho_d\rho_{c_{ij}})=\diag(\eta,\eta)\cdot(\rho_{d c_{ij}}) .
\end{equation}
We now repeat the first part of the proof with $(d c_{ij})$ instead
of $(m_{ij})$. This gives the projectivity
\begin{equation}\label{e:struttura}
     E(\alpha,\beta)\mapsto  E(\eta^{-1}\alpha\eta,\eta^{-1}\beta\eta)\cdot (\rho_{dc_{ij}})
    \mbox{~for all~} E(\alpha,\beta) \in\PG(1,E),
\end{equation}
which obviously coincides with $\pi$.
\end{proof}
The collineation $\kappa$ from the previous proposition can be extended in
precisely $\theta_{t-1}$ different ways to a projectivity of $\PG(1,E)$.
Even though this could be derived easily from well known results about spreads \cite[Sect.~1]{Knarr1995}, we give a short direct proof.
\begin{proposition}\label{p:hans3}%
There are precisely $\theta_{t-1}$ projectivities of $\PG(1,E)$ that fix
$\PG(1,F)$ pointwise.
\end{proposition}
\begin{proof}
Let $\pi$ be a projectivity of $\PG(1,E)$ that fixes $\PG(1,F)$ pointwise. We
repeat the second part of the proof of Prop.~\ref{p:hans1} under this stronger
assumption, while maintaining all notations from there. However, in our current
setting we may choose $(c_{ij})=\diag(1,1)\in\GL_2(q^t)$. We are thus led to
$(\pi_{ij})=\diag(\delta,\delta)$ for some $\delta\in E^*$. This $\delta$ has
to satisfy now $\delta^{-1}\rho_b\delta=\rho_b$ for all $b\in\Fqt$, which in
turn gives that $\eta$ is the trivial automorphism of $\Fqt$. Letting
$d:=1^\delta$, as we did before, gives therefore that $\pi$ is given by the
matrix
\begin{equation*}
    (\pi_{ij})=\diag(\rho_d,\rho_d) \mbox{~with~} d\in\Fqt^*.
\end{equation*}
\par
Conversely, for all $d\in\Fqt^*$ the matrix $\diag(\rho_d,\rho_d)$ determines a
projectivity of $\PG(1,E)$ that fixes $\PG(1,F)$ pointwise. Two matrices of
this kind give rise to the same projectivity if, and only if, they differ by a
factor $\diag(\gamma,\gamma)$, where $\gamma$ is an invertible element from the
centre of $E$. This condition for $\gamma$ is equivalent to $\gamma=\rho_f$
with $f\in\F_q^*$, whence the assertion follows from
$\theta_{t-1}=(q^t-1)/(q-1)=(\#\Fqt^*)/(\#\Fq^*)$.
\end{proof}

\begin{proposition}\label{p:varie_1}
  If $\pi\in\PGL_2(E)$ stabilizes $\PG(1,F)$, then $L_{T^\pi}=(L_T)^\pi$ and
  $L'_{T^\pi}=(L'_T)^\pi$ for each $T\in\PG(1,E)$.
\end{proposition}
\begin{proof}
  The first equation can be derived as follows:
  \begin{eqnarray*}
    L_{T^\pi}&=&\left\{P\in\PG(1,F)\mbox{ s.t. }P\nodist T^\pi\right\}=
    \left\{P^\pi\mbox{ s.t. }P\in\PG(1,F),\,P^\pi\nodist T^\pi\right\}\\
    &=&\left\{P^\pi\mbox{ s.t. }P\in\PG(1,F),\,
    P\nodist T\right\}=L_T^\pi.
  \end{eqnarray*}
  Let $h\in\Fqt^*$. Taking into account the structure of the projectivity
  $\pi$, described in (\ref{e:struttura}), $T^\pi h=(Th^{\eta^{-1}})^\pi$,
  whence $L'_{T^\pi}=(L'_T)^\pi$.
\end{proof}

\begin{remark}
  The investigation of scattered points can be restricted taking into account
  that if $T$ is a scattered point of $\PG(1,E)$, then there exist a
  projectivity $\pi$ of $\PG(1,E)$ and an element $\beta\in E^*$, such that
  $\PG(1,F)^\pi=\PG(1,F)$, $1\in\Spec(\beta)$, and $T^\pi=E(\one,\beta)$ (cf.
  \cite[Rem.~4.2]{LuMaPoTr2014}).
\end{remark}

\begin{lemma}\label{l:hans1}%
Let $s$ be the greatest element of $\{1,2,\ldots,t-1\}$ that divides $t$. Then
any two distinct $F $-chains of $\PG(1,E)$ have at most $q^s+1$ common points.
\end{lemma}
\begin{proof}
Due to $3=2^1+1\leq q^s+1 $ it is enough to consider two distinct $F$-chains
$C_1$ and $C_2$ with at least three distinct common points. By Prop.~\ref{p:0},
we may assume these points to be $E(\one,0)$, $E(0,\one)$, $E(\one,\one)$ and
$C_1=\PG(1,F )$. Applying Prop.~\ref{p:0} once more shows that there is a
projectivity $\pi$ of $\PG(1,E)$ that fixes each of the points $E(\one,0)$,
$E(0,\one)$, $E(\one,\one)$ and takes $C_1$ to $C_2$. Consequently, there is a
$\delta\in E^*$ such that
\begin{equation*}
     E(\alpha,\beta)^\pi=E(\alpha,\beta)\diag(\delta,\delta)=E(\delta^{-1}\alpha\delta,\delta^{-1}\beta\delta)
    \mbox{~for all~}  E(\alpha,\beta)\in\PG(1,E).
\end{equation*}
This gives
\begin{equation*}
    {C_1\cap C_1^\pi=\{ E(\one,\rho)\mid \rho\in F\cap (\delta^{-1}F \delta) \}\cup\{ E(0,\one)\}.}
\end{equation*}
The intersection $F \cap\delta^{-1}F \delta$ is a proper subfield of $F $,
since $\delta^{-1}F \delta$ is an isomorphic copy of $F $ in $E$ and $C_1\neq C_2$
implies $F \neq \delta^{-1}F \delta$. This gives $\#(C_1\cap C_2)= \#(F
\cap\delta^{-1}F \delta)+1\leq q^s+1$.
\end{proof}

\begin{proposition}\label{p:hans2}%
Let $T$ be a scattered point of $\PG(1,E)$ and assume that $t\geq 3$. Then
$L_T$ is contained in no $F $-chain other than $\PG(1,F )$.
\end{proposition}
\begin{proof}
From $t\geq 3$ follows $\#L_T=\theta_{t-1}> q^{t-1}+ 1$. The assertion is now
immediate from Lemma~\ref{l:hans1}.
\end{proof}

\begin{remark} For $t=2$ the set $L_T^\Psi$ is a regulus in $\PG(3,q)$ and
the $F $-chains are precisely the regular spreads in $\PG(3,q)$. Choose a point
that is off the hyperbolic quadric $\cH$ that carries $L_T^\Psi$. Then there are
as many regular spreads through $L_T^\Psi$ as there are external lines to $\cH$
through the chosen point. A straightforward counting shows that the number of
these lines is $\frac{1}{2}(q^2-q)$. Thus, unless $q=2$, there is more than one
$F $-chain through $L_T$.
\end{remark}

\begin{sloppypar}
\begin{theorem}\label{t:preq}
  Let $T$ and $U $ be points of $\PG(1,E)$, with $T$ a scattered point. Then
  the following assertions are equivalent:
  \begin{enumerate}[{\emph(}i{\emph)}]
  \item A collineation ${\kappa}\in\PGaL_{2}(q^t)$ with companion
      automorphism $\eta\in\gal$ exists, such that $\cB(T)^{{\kappa}}=\cB(U
      )$;
  \item $L_T$ and $L_{U }$ are projectively equivalent in $\PG(1,E)$.
  \end{enumerate}
\end{theorem}
\end{sloppypar}
\begin{proof}
If $\kappa$ is given as in $(i)$ then, by the first part of
Prop.~\ref{p:hans1}, there is a projectivity of $\PG(1,E)$ that takes
$\cB(T)^\iota=L_T$ to $\cB(U )^{\iota}=L_{U }$.

Conversely, let $\pi$ be a projectivity of $\PG(1,E)$ that takes $L_T$ to
$L_{U }$. There are two cases:

\emph{Case 1: $t\geq 3$}. From $\# L_T = \# L_{U }$ the point $U $ is
scattered. By Prop.~\ref{p:hans2}, each of the sets $L_{T}$ and $L_{U }$ is
contained in no $F $-chain other than $\PG(1,F )$. Hence $\PG(1,F )$ is
invariant under $\pi$ so that the second part of Prop.~\ref{p:hans1} shows the
existence of a collineation of $\PG(1,q^t)$ with the required properties.

\emph{Case 2: $t=2$}. The sets $\cB(T)$ and $\cB(U )$ are two linear sets of
rank $2$ and cardinality $q+1$, i.e.\ two Baer sublines of $\PG(1,q^2)$. These
are well known to be projectively equivalent.
\end{proof}

\begin{proposition}
  If $g\neq0$ and $t,q>3$, the sets $L_{T_0}$ and $L_{T_1}$, where $T_0$ and
  $T_1$ are defined in Examples~\emph{\ref{T_0}} and \emph{\ref{T_1}}, are not
  projectively equivalent in $\PG(1,E)$.
\end{proposition}
\begin{proof}
  The sets $\cB(T_0)$ and $\cB(T_1)$ are not projectively equivalent
  \cite[Example 4.6]{LuMaPoTr2014}. Since any linear set in the
  $\PGaL_{2}(q^t)$-orbit of a linear set of pseudoregulus type is again of
  pseudoregulus type, $\cB(T_0)$ and $\cB(T_1)$ also are not equivalent up to
  collineations. Then the assertion follows from Thm.~\ref{t:preq}.
\end{proof}

\section{Characterization of the linear sets of pseudo\-regulus type}

\begin{proposition}\label{p:zero-converse}
  \emph{\cite{LaShZa2015}} Let $T\in\PG(1,E )$ be such that $L_T$ is a set of
  pseudoregulus type. Then there is a $\varphi\in\PGL_2(E)$ such that
  $L_T^\varphi=L_T'$.
\end{proposition}
\begin{proof}
  Since up to projectivities in $\PG(1,q^t)$ there is a unique linear set of
  pseudoregulus type \cite{DoDu2014,LuMaPoTr2014}, it may be assumed that $T=
  E(\one,\tau)$ with $\tau$ a generator of $\gal$. A projectivity $\varphi$
  satisfying the thesis is given by the matrix $\diag(\one,\tau)\in\GL_2(E)$.
  As a matter of fact, for any $u\in\Fqt^*$,
  \[
     E(\rho_{1/u^\tau},\tau\rho_{1/u^\tau})=E(\one,\rho_{u^{\tau}}\tau\rho_{1/u^\tau})=
     E(\one,\tau\rho_{u^{\tau^2}/u^\tau})=E(\one,\rho_{u^\tau/u})^\varphi,
  \]
  and by Prop.~\ref{p:varie} this implies that $\varphi$ maps $L_T$ onto
  $L'_T$.
\end{proof}
The goal of this section is to prove the converse of
Prop.~\ref{p:zero-converse}.

\begin{proposition}\label{p:caratt-massimali} Let
$\varepsilon:\Pi_1\times\Pi_2\rightarrow\Pi_3$ be a projective embedding, where
$\Pi_j$ is a projective space of finite dimension $d_j\geq 1$ for $j=1,2,3$.
Let $\cU_1$ be the set of all $d_1$-subspaces of type $(\Pi_1\times
Q)^\varepsilon$ for $Q$ a point in $\Pi_2$, and $\cU_2$ the set of all
$d_2$-subspaces of type $(P\times \Pi_2)^\varepsilon$ for $P$ a point in
$\Pi_1$. Fix any point $S\in\Pi_2$ and consider the subspace $(\Pi_1\times
S)^\varepsilon\in\cU_1$. Under these assumptions the following assertions hold:
    \begin{enumerate}[1.]
  \item A line $y$ of\/ $\Pi_3$ is contained in some subspace belonging to
      $\cU_1$ if, and only if, there is a $d_2$-regulus $\cR_y$ in $\Pi_3$
      subject to the following three conditions:
  \begin{enumerate}[(a)]
    \item $\cR_y$ has a transversal line in $(\Pi_1\times
        S)^\varepsilon$.
    \item $\cR_y$ is contained in $\cU_2$.
    \item $y$ is a transversal line of $\cR_y$.
  \end{enumerate}
\item $A$ subspace $V$ of\/ $\Pi_3$ belongs to $\cU_1$ if, and only if, $V$
    has dimension $d_1$ and for any line $y$ of $V$ there is a
    $d_2$-regulus $\cR_y$ subject to the conditions (a), (b), and (c) from
    above.
\end{enumerate}
\end{proposition}
\begin{proof}
Our reasoning will be based on the following three facts: $(i)$~due to the
injectivity of $\varepsilon$, each point in the image of $\varepsilon$ is
incident with a unique subspace from $\cU_1$ and a unique subspace from
$\cU_2$; $(ii)$~for any line $\ell\subset\Pi_1$ the set
\begin{equation}\label{e:regulus}
    \left\{ (P\times\Pi_2)^\varepsilon \mid P\in\ell \right\}\subset\cU_2
\end{equation}
is a $d_2$-regulus; $(iii)$~the transversal lines of this regulus are precisely
the lines of the form $(\ell\times Q)^\varepsilon$ with $Q$ varying in $\Pi_2$.

\emph{Ad 1.} Let $y$ be a line such that $y\subset (\Pi_1\times
R)^\varepsilon\in\cU_1$, where $R\in\Pi_2$. Since the restriction of
$\varepsilon$ to $\Pi_1\times R$ is a collineation onto the subspace
$(\Pi_1\times R)^\varepsilon$, there is a unique line $\ell_y\subset\Pi_1$ such
that $y=(\ell_y\times R)^\varepsilon$. By \eqref{e:regulus}, the $d_2$-regulus
$\cR_y:=\{ (P\times\Pi_2)^\varepsilon \mid P\in\ell_y \}$ satisfies \emph{(b)}.
Furthermore, both $y$ and $(\ell_y\times S)^\varepsilon\subset(\Pi_1\times
S)^\varepsilon $ are transversal lines of $\cR_y$, whence conditions \emph{(a)}
and \emph{(c)} are satisfied too.

Conversely, assume that for a line $y\subset\Pi_3$ there is a regulus $\cR_y$
satisfying the three conditions from above. By \emph{(a)}, there is a line
$\ell_y\subset\Pi_1$ for which $(\ell_y\times S)^\varepsilon$ is a transversal
line of $\cR_y$. Now \emph{(b)} implies that $\cR_y$ can be written as in
\eqref{e:regulus} with $\ell$ to be replaced with $\ell_y$. Thus \emph{(c)}
shows that there is a point $R\in\Pi_2$ such that $y=(\ell_y\times
R)^\varepsilon$. This in turn gives $y\subset \Pi_1\times R \in\cU_1$.

\emph{Ad 2.} Let $V\in\cU_1$. The dimension of $V$ obviously is $d_1$. Given
any line $y\subset V$ there is a regulus $\cR_y$ with the required properties
by the first part of the proposition.

For a proof of the converse, we fix a point $Y\in V$ and consider an arbitrary
line $y\subset V$ through $Y$. By the first part of the proposition, there is
at least one point $R\in\Pi_2$ such that $y\subset(\Pi_1\times
R)^\varepsilon\in\cU_1$. This implies $Y=(X,R)^\varepsilon$ for some point
$X\in\Pi_1$. The point $R$ does not depend on the choice of the line $y$ on
$Y$. Consequently, $V\subset (\Pi_1\times R)^\varepsilon$ and, due to $d_1$
being the dimension of $V$, these two subspaces are identical.
\end{proof}

Clearly, analogous statements hold, where the roles of $\cU_1$ and $\cU_2$ are
interchanged. We will refer to $\cU_1$ and $\cU_2$ in
Prop.~\ref{p:caratt-massimali} as to the \emph{maximal subspaces of the
embedded product space $(\Pi_1\times\Pi_2)^\varepsilon$}. Notice that such
subspaces are defined with respect to the embedding, since
$(\Pi_1\times\Pi_2)^\varepsilon$ may contain further maximal subspaces.
Prop.~\ref{p:caratt-massimali} implies:
\begin{proposition}\label{p:rekonstr-alt}
  Let $\cU_1$ and $V$ be a collection of $d_1$-subspaces and a $d_2$-subspace
  in a projective space $\Pi$, respectively, where $d_1$ and $d_2$ are positive
  integers. There exists at most one collection $\cU_2$ of $d_2$-subspaces of\/
  $\Pi$ with $V\in\cU_2$, and such that $\cU_1$ and $\cU_2$ are the collections
  of maximal subspaces of an embedded product space.
\end{proposition}

\begin{remark}
  By Prop.~\ref{p:caratt-massimali}, a point $P\in\PG(1,E)$ is in $L'_T$ if,
  and only if, no line of $P^\Psi$ is irregular as defined in \cite{LaVdV2015}
  with respect to the scattered subspace $T^\Psi$.
\end{remark}

\begin{theorem}\label{t:pseudoreg}
Let $T=E(\one,\beta)$ be a scattered point, where $\beta\in E^*$, let $t\geq 3$
and suppose that there exists a projectivity $\varphi\in\PGL_2(E)$ such that
$L_T^\varphi=L'_T$. Then $L_T$ is of pseudoregulus type.
\end{theorem}
\begin{proof}

We split the proof into four steps.

\emph{Step 1.} First we fix some matrix $(\varphi_{ij})\in\GL_2(E)$ that
describes $\varphi$. Then we choose any point $U\in\PG(1,E)$ such that
$U^\varphi\in L_T$. The point $U^\varphi$ is non-distant to all points of
$L'_T$, whence $U$ is non-distant to all (namely $\theta_{t-1}$) points of
$L_T$. So $U$ is scattered and $L_T=L_U$. Also, there is a $\gamma\in E^*$
satisfying
\begin{equation*}
    U =  E(\one,\gamma).
\end{equation*}
According to \eqref{e:hat-phi}, the matrix $(\varphi_{ij})$ describes also that
projective collineation $\hat\varphi$ of $\PG_q(\Fqt^2)$ whose action on the
Grassmannian $\PG(1,E)^\Psi$ coincides with $\Psi^{-1}\varphi\Psi$. It can be
deduced from Prop.~\ref{p:PGxPG-embed} that for any scattered point
$X\in\PG(1,E )$, $L_X^{\Psi}$ and $(L'_X)^\Psi$ are the two collections of
maximal subspaces of an embedded product space. So, a repeated application of
Prop.~\ref{p:rekonstr-alt} implies:
\begin{itemize}
\item [(a)] $L_T^\Psi$ is the unique collection of maximal subspaces of an
    embedded product space, containing $U^{\varphi\Psi}$, the other
    collection being $(L'_T)^\Psi$.

\item [(b)] $(L'_U)^\Psi$ is the unique collection of maximal subspaces of
    an embedded product space, containing $U^{\Psi}$, the other one being
$L^\Psi_U=L^\Psi_T$.
\end{itemize}
By applying $\hat\varphi$ to (b), one obtains:
\begin{itemize}
\item[(c)] $(L'_U)^{\varphi\Psi}$ is the unique collection of maximal
    subspaces of an embedded product space, containing $U^{\varphi\Psi}$,
    the other one being $L^{\varphi\Psi}_T=(L'_T)^\Psi$.
\end{itemize}
By (a) and (c)
\begin{equation}\label{e:gleich}
    (L'_U)^\varphi = L_T = L_U .
\end{equation}

\emph{Step 2.} Let $H$ be the subgroup $\GL_2(E)$ formed by all matrices
$\diag(\rho_h,\rho_h)$ with $h\in\Fqt^*$. Also, let $\Lambda$ be the associated
group of projectivities of $\PG(1,E)$. Both $H\cong\Fqt^*$ and $\Lambda\cong
{\Fqt^* / \Fq^*}$ are cyclic. Now, for all $h,k\in\Fqt^*$ the equality
$E(\rho_h,\gamma\rho_h)=E(\rho_k,\gamma\rho_k)$ holds precisely when $h$ and
$k$ are $\Fq$-linearly dependent (cf.\ Prop.~\ref{p:Th}). Consequently we have
shown: the cyclic group $\Lambda$ acts regularly on $L'_U$ and fixes $L_U=L_T$
pointwise.
\par
Consider now the subgroup
\begin{equation*}
    H':=(\varphi_{ij})^{-1}\cdot H\cdot (\varphi_{ij})
\end{equation*}
of $\GL_2(E)$ and the corresponding group $\Lambda':=\varphi^{-1}\Lambda\varphi$ of
projectivities. By the above, the group $\Lambda'$ is cyclic, acts regularly on
$L_T$, and fixes $L'_T$ pointwise. From $t\geq 3$ and Prop.~\ref{p:hans2} the
$F$-chain $\PG(1,F)$ is invariant under the action of $\Lambda'$. Since $\Lambda'$ has
order $\theta_{t-1}$, which is the size of the orbit $L_T$, the group $\Lambda'$
acts faithfully on $\PG(1,F)$.
\par
\emph{Step 3.} Let us choose any element $h\in\Fqt^*$. The projectivity $\pi$
given by the matrix
\begin{equation*}
    (\pi_{ij}):= (\varphi_{ij})^{-1}\cdot \diag(\rho_h,\rho_h)\cdot(\varphi_{ij})\in\GL_2(E)
\end{equation*}
fixes $\PG(1,F)$, as a set. We therefore can repeat the second part of the
proof of Prop.~\ref{p:hans1} up to \eqref{e:mu_final}. By this
formula\footnote{Take notice that the elements $c_{ij}$ that are used now play
the role of the elements $dc_{ij}$ that appear in \eqref{e:mu_final}.}, there
exists an automorphism $\eta\in\Gal(\Fqt/\Fq)$ and an invertible matrix
$(c_{ij})\in\GL_2(q^t)$ such that
\begin{equation}\label{e:mit_eta}
    (\pi_{ij})=
    \diag(\eta,\eta)\cdot(\rho_{c_{ij}}) \in H'.
\end{equation}
\emph{We claim that $\eta=\one$}. In order to verify this assertion, we fix any
element $u\in\Fqt^*$. The point $E(\rho_u,\rho_{u^\beta})$ is in $L_T$, which
is invariant under $\pi\in\Lambda'$ by Step~2. This implies that there is a
$v\in\Fqt^*$ such that $E(\rho_u,\rho_{u^\beta})^\pi =
 E(\rho_v,\rho_{v^\beta})$. Furthermore, $L'_T$ is fixed pointwise under $\pi$.
So, the associated projective collineation $\hat\pi$ of $\PG_q(\Fqt^2)$ sends
for all $k\in\Fqt^*$ the point
\begin{equation*}
     \la(uk,u^\beta k)\ra_q  =  E(\rho_u,\rho_{u^\beta})^\Psi\cap  E(\rho_k,\beta\rho_k)^\Psi
\end{equation*}
to the point
\begin{equation*}
      E(\rho_v,\rho_{v^\beta})^\Psi\cap  E(\rho_k,\beta\rho_k)^\Psi = \la(vk,v^\beta k)\ra_q.
\end{equation*}
Therefore, for each $k\in\Fqt^*$ there is an element $x_k\in\Fq^*$ such that
\begin{equation*}
    ((u k)^\eta,(u^{\beta} k)^\eta)\cdot(c_{ij}) = x_k(vk,v^\beta k).
\end{equation*}
As $\Fqt^2$ is also a vector space over $\Fqt$, this can be rewritten in the
form
\begin{equation*}
    k^\eta (u ^\eta,u^{\beta\eta})\cdot(c_{ij}) = x_k k (v,v^\beta ).
\end{equation*}
Letting $k:=1$ gives $(u ^\eta,u^{\beta\eta})\cdot(c_{ij}) = x_1 (v,v^\beta)$
so that $ k^\eta x_1 (v,v^\beta) = x_k k (v,v^\beta )$. We thus have arrived at
\begin{equation*}
    k^\eta = (x_k / x_1) k \mbox{~for all~} k\in\Fqt^*.
\end{equation*}
This means that each $k\in\Fqt^*$ is an eigenvector of the $\Fq$-linear mapping
$\eta\in E$. Thus $\eta$ has a unique eigenvalue, i.e., $x_k/x_1$ is a constant
that does not depend on $k$. Finally, $1^\eta=1$ shows that this constant is
equal to $1$.

\emph{Step 4.} We maintain all notions from the previous step (with
$\eta=\one)$, under the additional assumption that $h$ is a generator of the
multiplicative group $\Fqt^*$. There are three possibilities for the matrix
$(c_{ij})\in\GL_2(q^t)$:

\par
First, assume that $(c_{ij})$ has a single eigenvalue in $\Fqt$. Then, since
$(c_{ij})$ cannot be a scalar multiple of the identity matrix, it is similar to
a matrix of the form
      \begin{equation*}
        a\begin{pmatrix}
        1 & b\\0 & 1
        \end{pmatrix}
        \mbox{~ with~} a,b\in\Fqt^*.
      \end{equation*}
Due to $\Char \Fqt = p$, the $p$-th power of this matrix is $\diag(a^p,a^p)$,
which gives that $\Lambda'$ acts on $\PG(1,F)$ as a cyclic permutation group of
order $p$. Since $L_T$ is an orbit under this action, we obtain the
contradiction $\#L _T=\theta_{t-1} \leq p$.

\par
Next, assume that $(c_{ij})$ has no eigenvalue in $\Fqt$, whence it has two
distinct eigenvalues $z\neq \overline z$ in $\Fqtt\supset\Fqt$. Here
$\overline{\phantom{I}}$ denotes the unique non-trivial automorphism in
$\Gal(\Fqtt/\Fqt)$. So, over $\Fqtt$, the matrix $(c_{ij})$ is similar to
$\diag(z,\overline z)$. Let $w$ be a generator of the multiplicative group
$\Fqtt^*$. We consider the matrix group $\{\diag(w^s,\overline w^s)\mid
s=1,\ldots,q^{2t}-1\}$ of order $q^{2t}-1$. It acts as a group $\Omega$ of
projectivities on $\PG(1,q^{2t})$. The kernel of this action comprises all
matrices $\diag(w^s,\overline w^s)$ with $w^s=\overline w^s = \overline{w^s}$
or, equivalently, with $w^s\in\Fqt^*$. Hence we have
\begin{equation*}
  \#\Omega = \frac{q^{2t}-1}{q^t-1} = q^{t}+1.
\end{equation*}
The powers of $\diag(z,\overline z)$ constitute a matrix group and give rise to
a subgroup of $\Omega$. This subgroup is isomorphic to $\Lambda'$, and
therefore its order is $\theta_{t-1}$. This gives that $\theta_{t-1}$ divides
$q^{t}+1$, which is impossible due to
\begin{equation*}
    (q-1)\,\theta_{t-1} = q^{t}-1<q^{t}+1 < q^t+q^{t-1}+\cdots+q = q\, \theta_{t-1} .
\end{equation*}

\par
By the above, $(c_{ij})$ has two distinct eigenvalues $a,b$ in $\Fqt^*$ and so
there is a matrix $(m_{ij})\in\GL_2(q^t)$ such that
\begin{equation*}
     \diag(a,b) =  (m_{ij})^{-1} \cdot (c_{ij}) \cdot (m_{ij}) .
\end{equation*}
Let $(\mu_{ij})=(\rho_{m_{ij}})\in\GL_2(E)$ and let $\mu\in\PGL_2(E)$ be the
projectivity given by $(\mu_{ij})$. It fixes the $F$-chain $\PG(1,F)$ as a
set. We therefore obtain that $T^\mu$ is a scattered point and the equality
$L_T^\mu = L_{T^\mu}$ (cf.\ Prop.~\ref{p:varie_1}). The group
\begin{equation*}
    H'':=(\mu_{ij})^{-1} \cdot H' \cdot (\mu_{ij})
    = \{ \diag(\rho_a,\rho_b)^s\mid s=1,2,\ldots,q^t-1 \}
\end{equation*}
induces the cyclic group $\Lambda'':=\mu^{-1}\Lambda'\mu$ of projectivities, which
acts regularly on $L_T^\mu$ and fixes the points $E(\one,0)$ and $E(0,\one)$.
So $L_T^\mu$ contains none of these points, whence we have
\begin{equation*}
 T^\mu=  E(\one,\delta) \mbox{~with~}\delta\in  E^*.
\end{equation*}
The matrix group generated by $\diag(\one,\rho_{b/a})$ also induces the group
$\Lambda''$. So the order of $b/a$ in the multiplicative group $\Fqt^*$ is
$\theta_{t-1}$. There is an element $e\in\Fqt^*$ such that
\begin{equation*}
    b/a=e^{q-1} .
\end{equation*}
This allows us to give another group of matrices that induces $\Lambda''$, namely
the group generated by
\begin{equation}\label{e:galois}
    \rho_e\diag(\one,\rho_{b/a}) = \diag(\rho_e,\rho_{e^\sigma}),
\end{equation}
where $\sigma\in\Gal(\Fqt/\Fq)$ is given by $x\mapsto x^{q}$.

Finally, let $d:=1^\delta$. Then $E(\one,\rho_d)\in L_T^\mu$ and, by writing
the action of $\Lambda''$ in terms of the powers of the matrix from
\eqref{e:galois}, we obtain
\begin{equation*}
    L_T^\mu =
    \{  E(\one,\rho_d)\diag( \rho_u,\rho_{u^\sigma})
    =E(\rho_u,\rho_{u^\sigma d})
    \mid u=e^s,\;s=1,2,\ldots,\theta_{t-1} \} .
\end{equation*}
Thus $L_T^\mu$ is contained in $L_W$, where
\begin{equation*}
   W :=  E(\one,\sigma\rho_d).
\end{equation*}
This implies, due to $\#L_T^\mu=\#L_W$, that $L_T$ is projectively equivalent
to $L_W$, which is of pseudoregulus type.
\end{proof}

It should be noted that in Step~4 in the proof above the choice of a generator
$h$ of the multiplicative group $\Fqt^*$ serves the sole purpose that the
projectivity related to the matrix $(\rho_{c_{ij}})$ is a generator of the
group $\Lambda'$ acting regularly on $L_T$. Bearing this in mind, the following
proposition can be extrapolated:
\begin{proposition}\label{p:gruppo-transitivo}
  Let $L$ be a scattered linear set of rank $t$ in $\PG(1,q^t)$. If there
  exists a cyclic subgroup of $\PGL_2(q^t)$ acting regularly on $L$, then $L$
  is of pseudoregulus type.
\end{proposition}
On the other hand, any linear set of pseudoregulus type is projectively
equivalent to $\{\la(1,u^{q-1})\ra_{q^t}\mid u\in\Fqt^*\}$ \cite{DoDu2014},
\cite{LuMaPoTr2014}. Hence the converse of Prop.~\ref{p:gruppo-transitivo} can
be proved directly.

\noindent
Authors' addresses:\\[1mm]
\noindent
Hans Havlicek\\
Institut f\"{u}r Diskrete Mathematik und Geometrie,\\
Technische Universit\"{a}t Wien,\\
Wiedner Hauptstra{\ss}e 8--10,\\
A-1040 Wien,\\
Austria\\[1mm]
\noindent
Corrado Zanella\\
Dipartimento di Tecnica e Gestione dei Sistemi Industriali,\\
Universit\`a di Padova,\\
Stradella S.\ Nicola, 3,\\
I-36100 Vicenza,\\
Italy


\begin{thebibliography}{pippo}\itemsep0pt plus 0.5pt


\bibitem{Bl1999} {\sc A. Blunck:} Regular spreads and chain geometries. Bull.\
    Belg.\ Math.\ Soc.\ \textbf6 (1999), 589--603.

\bibitem{BlHa2000a} {\sc A. Blunck - H. Havlicek:} Extending the concept of
    chain geometry, Geom.\ Dedicata, \textbf{83} (2000), 119--130.

\bibitem{BlHe2005} {\sc A. Blunck - A. Herzer:} Kettengeometrien -- {E}ine
    {E}inf\"{u}hrung, Shaker Verlag, Aachen, 2005.

\bibitem{CsZa201x} {\sc B. Csajb\'ok - C. Zanella:} On the equivalence of
    linear sets. Des.\ Codes Cryptogr., DOI:10.1007/s10623-015-0141-z.

\bibitem{CsZa2016} {\sc B. Csajb\'ok - C. Zanella:} On scattered linear sets of
    pseudoregulus type in $\PG(1,q^t)$. arXiv:1506.08875v1

\bibitem{DoDu2014} {\sc G. Donati - N. Durante:} Scattered linear sets
    generated by collineations between pencils of lines, J. Algebraic Combin.\
    \textbf{40} (2014), 1121--1134.

\bibitem{Dy1991} {\sc R.H. Dye:} Spreads and classes of maximal subgroups of
    $\GL_n(q)$, $\SL_n(q)$, $\PGL_n(q)$ and $\PSL_n(q)$. Ann.\ Mat.\ Pura
    Appl.\ (4) \textbf{158} (1991), 33--50.

\bibitem{Ha2012} {\sc H. Havlicek:} Divisible designs, {L}aguerre geometry, and
    beyond, J.\ Math.\ Sci.\ (N.Y.)
   \textbf{186} (2012), 882--926.


\bibitem{He1995} {\sc A. Herzer:} Chain geometries, In: F. Buekenhout (ed.),
    Handbook of Incidence Geometry, Elsevier, Amsterdam, 1995, 781--842.

\bibitem{Knarr1995}
    {\sc N. Knarr:} Translation Planes, Lecture Notes in Mathematics,
       1611, Springer, Berlin, 1995.


\bibitem{Lang93} {\sc S. Lang:} Algebra, 3rd ed., Addison-Wesley, Reading, MA,
    1993.

\bibitem{LaShZa2015} {\sc M. Lavrauw - J. Sheekey - C. Zanella:} On embeddings
    of minimum dimension of $\mathrm{PG}(n,q)\times \mathrm{PG}(n,q)$, Des.\
    Codes Cryptogr.\ \textbf{74} (2015), 427--440.

\bibitem{LaVdV2015} {\sc M. Lavrauw - G. Van de Voorde:} On linear sets on a
    projective line. Des.\ Codes Cryptogr. \textbf{56} (2010), 89--104.


\bibitem{LaVa2015} {\sc M. Lavrauw - G. Van de Voorde:} Field reduction and
    linear sets in finite geometry, Topics in finite fields, 271--293,
    Contemp.\ Math., 632, Amer.\ Math.\ Soc., Providence, RI, 2015.

\bibitem{LaZa2015} {\sc M. Lavrauw - C. Zanella:} Subgeometries and linear sets
    on a projective line. Finite Fields Appl.\ \textbf{34} (2015), 95--106.

\bibitem{LaZa201X} {\sc M. Lavrauw - C. Zanella:} Subspaces intersecting each
    element of a regulus in one point, Andr\'{e}-Bruck-Bose representation and
    clubs. Electron. J. Combin. \textbf{23} (2016), Paper 1.37, 11 pp.

\bibitem{LuMaPoTr2014} {\sc G. Lunardon - G. Marino - O. Polverino - R.
    Trombetti:} Maximum scattered linear sets of pseudoregulus type and the
    Segre variety ${\cal S}_{n,n}$. J.\ Algebraic Combin.\ \textbf{39} (2014),
    807--831.

\bibitem{LuPo2001} {\sc G. Lunardon - O. Polverino:} Blocking sets and
    derivable partial spreads. J.\ Algebraic Combin.\ \textbf{14} (2001),
    49--56.

\bibitem{Po2010} {\sc O. Polverino:} Linear sets in finite projective spaces,
    Discrete Math.\ {310}, 3096--3107 (2010).

\bibitem{Wan1996} {\sc Z.-X. Wan:} Geometry of Matrices, World Scientific,
    Singapore, 1996.


\end{thebibliography}
\end{document}